\documentclass[a4paper,11pt,fleqn]{article}
\usepackage{pstricks}
\usepackage{amsmath}
\usepackage{amssymb}
\usepackage{mathrsfs} 
\usepackage{theorem} 
\usepackage{euscript}
\usepackage{exscale,relsize}
\usepackage{graphicx}
\usepackage{srcltx}
\usepackage{xcolor}
\usepackage{charter}
\usepackage{url}
\PassOptionsToPackage{normalem}{ulem}
\usepackage{ulem}

\renewcommand{\leq}{\ensuremath{\leqslant}}
\renewcommand{\geq}{\ensuremath{\geqslant}}

\newcommand{\pair}[2]{\langle{{#1},{#2}}\rangle} 
\newcommand{\Pair}[2]{\big\langle{{#1},{#2}}\big\rangle}

\newcommand{\menge}[2]{\big\{{#1}~\big |~{#2}\big\}}

\newcommand{\IDD}{\ensuremath{\text{\rm int\:dom}f}}

\newcommand{\IDDS}{\ensuremath{\text{\rm int\:dom}f^*}}

\newcommand{\RX}{\ensuremath{\left]-\infty,+\infty\right]}}

\newcommand{\XX}{\ensuremath{{\mathcal X}}}
\newcommand{\HH}{\ensuremath{{\mathcal H}}}

\newcommand{\emp}{\ensuremath{{\varnothing}}}
\newcommand{\prox}{\ensuremath{\text{\rm prox}}}

\newcommand{\Id}{\ensuremath{\operatorname{Id}}\,}

\newcommand{\RR}{\ensuremath{\mathbb{R}}}

\newcommand{\RP}{\ensuremath{\left[0,+\infty\right[}}

\newcommand{\RPP}{\ensuremath{\left]0,+\infty\right[}}
\newcommand{\NN}{\ensuremath{\mathbb N}}

\newcommand{\weakly}{\ensuremath{\:\rightharpoonup\:}}
\newcommand{\exi}{\ensuremath{\exists\,}}
\newcommand{\pinf}{\ensuremath{{+\infty}}}

\newcommand{\dom}{\ensuremath{\text{\rm dom}\,}}

\newcommand{\Cart}{\ensuremath{\raisebox{-0.5mm}{\mbox{\huge{$\times$}}}}}

\newcommand{\inte}{\ensuremath{\text{\rm int}\,}}

\newcommand{\ran}{\ensuremath{\text{\rm ran}\,}}

\newcommand{\gra}{\ensuremath{\text{\rm gra}\,}}

\newcommand{\argmin}{\ensuremath{\text{\rm argmin}\,}}

\newcommand{\Fix}{\ensuremath{\text{\rm Fix}\,}}

\newcommand{\Argmin}{\ensuremath{\text{\rm Argmin}\,}}

\newcommand{\BP}{\ensuremath{\EuScript P}}

\newcommand{\BF}{\ensuremath{\EuScript F}}
\newcommand{\BFS}{\ensuremath{\EuScript S}}

\newtheorem{theorem}{Theorem}[section]
\newtheorem{lemma}[theorem]{Lemma}
\newtheorem{corollary}[theorem]{Corollary}
\newtheorem{proposition}[theorem]{Proposition}

\theoremstyle{plain}{\theorembodyfont{\rmfamily}%
}
\theoremstyle{plain}{\theorembodyfont{\rmfamily}%
\newtheorem{example}[theorem]{Example}}
\theoremstyle{plain}{\theorembodyfont{\rmfamily}%
\newtheorem{remark}[theorem]{Remark}}
\theoremstyle{plain}{\theorembodyfont{\rmfamily}%
}
\theoremstyle{plain}{\theorembodyfont{\rmfamily}%
\newtheorem{condition}[theorem]{Condition}}
\theoremstyle{plain}{\theorembodyfont{\rmfamily}%
\newtheorem{definition}[theorem]{Definition}}
\theoremstyle{plain}{\theorembodyfont{\rmfamily}
}
\theoremstyle{plain}{\theorembodyfont{\rmfamily}
}

\numberwithin{equation}{section}
\begin{document}

\title{\sffamily\LARGE Variable Quasi-Bregman Monotone Sequences}

\author{Quang Van Nguyen\\[5mm]
\small
\small Sorbonne Universit\'es -- UPMC Univ. Paris 06\\
\small UMR 7598, Laboratoire Jacques-Louis Lions\\
\small F-75005 Paris, France\\
\small {\ttfamily quangnv@ljll.math.upmc.fr}
}
\date{\sffamily ~}
\maketitle
\setcounter{page}{1}

\vskip 8mm

\begin{abstract}
We introduce a notion of variable quasi-Bregman monotone sequence 
which unifies the notion of variable metric quasi-Fej\'er monotone 
sequences and that of Bregman monotone sequences. The results 
are applied to analyze the asymptotic behavior of proximal
iterations based on variable Bregman distance and of algorithms for
solving convex feasibility problems in reflexive real Banach spaces. 
\end{abstract}

\newpage

\section{Introduction}
\noindent 
The concept of Fej\'er monotonicity and its variants plays an 
important role in the convergence analysis of many fixed point 
and optimization algorithms in Hilbert spaces 
\cite{IIBaus96,IIBC11_B,IIElse01,IIEncy09,IIErmo71,IIRaik67}. A 
recent development in this area is the extension of the notion 
of (quasi)-Fej\'er sequence to the case when the underlying 
metric is allowed to vary over the iterations \cite{IICV13b}. 
Since Fej\'er 
monotonicity is of limited use outside of Hilbert spaces, the 
notion of Bregman monotonicity was introduced in \cite{IIBBC03} 
to provide a unifying framework for the convergence analysis of 
various algorithms for solving nonlinear problems. The main 
objective of the present paper is to unify the 
work of \cite{IICV13b} on variable metric Fej\'er sequences and 
that of \cite{IIBBC03} on Bregman monotone sequences by 
introducing the notion of a variable quasi-Bregman monotone 
sequence and by investigating its asymptotic properties. 
We apply these results to a variable Bregman proximal 
point algorithm and to convex feasibility problems in Banach spaces. 
Our paper revolves around the following definitions.

\begin{definition}{\rm\cite{IIBBC01,IIBBC03}}
Let $\XX$ be a reflexive real Banach space, 
let $\XX^*$ be the topological dual space of $\XX$, 
let $\pair{\cdot}{\cdot}$ be the duality pairing between 
$\XX$ and $\XX^*$, let $f\colon\XX\to\RX$ be a 
lower semicontinuous convex function that is 
G\^ateaux differentiable on $\IDD\neq\emp$, 
let $f^*\colon\XX^*\to\RX\colon
x^*\mapsto\sup_{x\in\XX}(\pair{x}{x^*}-f(x))$ be 
conjugate of $f$, and let 
\begin{equation}
\label{IIe:subdiff}
\partial f\colon\XX\to 2^{\XX^*}\colon x\mapsto
\menge{x^*\in\XX^*}{(\forall y\in\XX)\,
\pair{y-x}{x^*}+f(x)\leq f(y)},  
\end{equation}
be Moreau subdifferential of $f$.
The \emph{Bregman distance} associated with $f$ is 
\begin{equation}
\label{IIe:Bdist}
\begin{aligned}
D^f\colon\XX\times\XX&\to\,\left[0,\pinf\right]\\
(x,y)&\mapsto 
\begin{cases}
f(x)-f(y)-\pair{x-y}{\nabla f(y)},&\text{if}\;\;y\in\IDD;\\
\pinf,&\text{otherwise}.
\end{cases}
\end{aligned}
\end{equation}
In addition, $f$ is a \emph{Legendre function} if it is 
\emph{essentially smooth} in the sense that $\partial f$ is 
both locally bounded and single-valued on its
domain, and \emph{essentially strictly convex} in the sense
that $\partial f^*$ is locally bounded on its domain and 
$f$ is strictly convex on every convex subset of $\dom\partial f$.
Let $\varphi\colon\XX\to\RX$ be a lower semicontinuous 
convex function which is bounded from below 
and $\dom\varphi\cap\IDD\neq\emp$. 
The \emph{$D^f$-proximal operator} of $\varphi$ is
\begin{equation}
\begin{aligned}
\label{IIe:dprox}
\prox_{\varphi}^f\colon\IDD&\to\dom\varphi\cap\IDD\\
y&\mapsto\underset{x\in\XX}{\text{argmin}}\,\varphi(x)+D^f(x,y).
\end{aligned}
\end{equation}
Let $C$ be a closed convex subset of $\XX$ such that
$C\cap\IDD\neq\emp$. The \emph{Bregman projector} onto 
$C$ induced by $f$ is 
\begin{equation}
\label{IIe:2001}
\begin{aligned}
P^f_C\colon\IDD&\to C\cap\IDD\\
y&\mapsto\underset{x\in C}{\text{argmin}}\,D^f(x,y),
\end{aligned}
\end{equation}
and the \emph{$D^f$-distance} 
to $C$ is the function
\begin{equation}
\begin{aligned}
D^f_C\colon\XX&\to\left[0,+\infty\right]\\
y&\mapsto\inf D^f(C,y).
\end{aligned}
\end{equation}
\end{definition}

The paper is organized as follows. In Section~\ref{IIsec:VBMS}, 
we introduce the notion of a variable quasi-Bregman monotone 
sequence and investigate its asymptotic properties. Basic results 
on $D^f$-proximal operators are reviewed in 
Section~\ref{IIsec:VBPA}. Applications to 
a variable Bregman proximal point algorithm 
and to the convex feasibility problem are considered 
in Section~\ref{IIsec:app}.

\noindent {\bf Notation and background.}
The norm of a Banach space is denoted by $\|\cdot\|$. 
The symbols $\weakly$ and $\to$ represent respectively weak and 
strong convergence. The set of weak sequential cluster points 
of a sequence $(x_n)_{n\in\NN}$ is denoted by 
$\mathfrak{W}(x_n)_{n\in\NN}$. Let $M\colon\XX\to 2^{\XX}$. 
The domain of $M$ is $\dom M=\menge{x\in\XX}{Mx\neq\emp}$, 
the range of $M$ is 
$\ran M=\menge{y\in\XX}{(\exi x\in\XX)\;y\in Mx}$, 
the graph of $M$ is 
$\gra M=\menge{(x,y)\in\XX\times\XX}{y\in Mx}$, 
and the set of fixed points of $M$ is 
$\Fix M=\menge{x\in\XX}{x\in Mx}$. 
A function $f\colon\XX\to\RX$ is coercive if 
$\lim_{\|x\|\to+\infty}f(x)=+\infty$. Denote by 
$\Gamma_0(\XX)$ the class of all lower 
semicontinuous convex functions $f\colon\XX\to\RX$ such that 
$\dom f=\menge{x\in\XX}{f(x)<+\infty}\neq\emp$. 
Let $f\in\Gamma_0(\XX)$. The set of global minimizers of 
a function $f$ is denoted by $\Argmin f$. 
In addition, if $f$ is G\^ateaux differentiable 
on $\IDD\neq\emp$ then
\begin{equation}
\begin{aligned}
\hat{f}\colon\XX&\to\RX\\
x&\mapsto
\begin{cases}
f(x),&\text{if}\;x\in\IDD;\\
\pinf,&\text{otherwise}.
\end{cases}
\end{aligned}
\end{equation}
Finally, $\ell_{+}^1(\NN)$ is the set of all 
summable sequences in $\left[0,+\infty\right[$.

\section{Variable Bregman monotonicity}
\label{IIsec:VBMS}

\begin{definition}
\label{IIdef:alpha}
Let $\XX$ be a reflexive real Banach space 
and let $f\in\Gamma_0(\XX)$ be G\^ateaux differentiable on 
$\IDD\neq\emp$. Then
\begin{equation}
\BF(f)=\menge{g\in\Gamma_0(\XX)}{g\;\text{is G\^ateaux 
differentiable on}\;\dom g=\IDD}.
\end{equation}
Moreover, if $g_1$ and $g_2$ are in $\BF(f)$, then
\begin{equation}
g_1\succcurlyeq g_2\quad\Leftrightarrow\quad
(\forall x\in\dom f)(\forall y\in\IDD)
\quad D^{g_1}(x,y)\geq D^{g_2}(x,y).
\end{equation}
For every $\alpha\in\RP$, set
\begin{equation}
\BP_{\alpha}(f)=\menge{g\in\BF(f)}{g\succcurlyeq\alpha f}.
\end{equation}
\end{definition}

\begin{remark}
In Definition~\ref{IIdef:alpha}, 
suppose that $\XX$ is a Hilbert space 
and let $\alpha\in\RPP$. Then the following hold:
\begin{enumerate}
\item
\label{IIle:1i}
Suppose that $f$ is Fr\'echet differentiable on $\XX$. Then 
$\|\cdot\|^2/2\in\BP_{\alpha}(f)$ if and only if 
$\nabla f$ is $\alpha^{-1}$-Lipschitz continuous.
\item
\label{IIle:1ii}
Let $\BFS(\XX)$ 
be the space of self-adjoint bounded linear operators from 
$\XX$ to $\XX$. The Loewner partial ordering on $\BFS(\XX)$ 
is defined by 
\begin{equation}
(\forall U_1\in\BFS(\XX))(\forall U_2\in\BFS(\XX))
\quad U_1\succcurlyeq U_2\quad\Leftrightarrow\quad
(\forall x\in\XX)\quad\Pair{x}{U_1x}\geq\Pair{x}{U_2x}.
\end{equation}
Set $\BP_{\alpha}(\XX)=\menge{U\in\BFS(\XX)}{U\succcurlyeq\alpha\Id}$. 
Let $U\in\BFS(\XX)$ and $V\in\BFS(\XX)$ be such that 
$V\succcurlyeq\alpha U$. Suppose that $f\colon x\mapsto\pair{x}{Ux}/2$ 
and $g\colon x\mapsto\pair{x}{Vx}/2$. 
Then $g\in\BP_{\alpha}(f)$.
\end{enumerate}
\end{remark}
\begin{proof}
\ref{IIle:1i}:
First, since $f$ is Fr\'echet differentiable, 
$\partial f=\nabla f$ \cite[Proposition~17.26]{IIBC11_B} 
and hence, by \cite[Corollary~16.24]{IIBC11_B}, 
$(\nabla f)^{-1}=(\partial f)^{-1}=\partial f^*$. 
Now, we have 
\begin{align}
\label{IIe:1}
\|\cdot\|^2/2&\in\BP_{\alpha}(f)
\Leftrightarrow\;(\forall x\in\XX)(\forall y\in\XX)\;
\|x-y\|^2/2\geq\alpha D^f(x,y)\nonumber\\
&\Leftrightarrow\;(\forall x\in\XX)(\forall y\in\XX)\;
\|x-y\|^2/(2\alpha)\geq f(x)-f(y)-\pair{x-y}{\nabla f(y)}
\nonumber\\
&\Leftrightarrow\;(\forall x\in\XX)(\forall y\in\XX)\;
f(x)\leq f(y)+\pair{x-y}{\nabla f(y)}+\|x-y\|^2/(2\alpha).
\end{align}
The assertion therefore follows by invoking 
\cite[Theorem~18.15]{IIBC11_B}.

\ref{IIle:1ii}:
We observe that $f$ and $g$ are G\^ateaux 
differentiable on $\XX$ 
with $\nabla f=U$ and $\nabla g=V$. Consequently,
\begin{align}
(\forall x\in\XX)(\forall y\in\XX)\quad 
D^g(x,y)&=\pair{x}{Vx}/2-\pair{y}{Vy}/2
-\pair{x-y}{Vy}\nonumber\\
&=\pair{x-y}{Vx-Vy}/2\nonumber\\
&\geq\alpha\pair{x-y}{Ux-Uy}/2\nonumber\\
&=\alpha D^f(x,y).
\end{align}
\end{proof}

\begin{example}
Let $\XX$ be a reflexive real Banach space, 
let $f\in\Gamma_0(\XX)$ be G\^ateaux differentiable 
on $\IDD\neq\emp$, let $\alpha\in\RP$, 
and let $g\in\Gamma_0(\XX)$ be G\^ateaux differentiable 
on $\dom g=\IDD$. 
Suppose that and $g-\alpha f$ is convex 
(which means that $g$ is more convex than $\alpha f$ 
in the terminology of J. J. Moreau \cite{IIMo65}). 
Then $g\in\BP_{\alpha}(f)$.
\end{example}
\begin{proof}
We first note that $\dom h=\IDD$. Since 
$f$ and $g$ are G\^ateaux differentiable on 
$\IDD$ by \cite[Proposition~3.3]{IIPhelps}, $h=g-\alpha f$ is 
likewise. Furthermore, 
\begin{equation}
(\forall x\in\dom f)(\forall y\in\IDD)\quad
D^g(x,y)-\alpha D^f(x,y)=D^h(x,y)\geq 0.
\end{equation}
\end{proof}

The following definition brings together the notions of Bregman 
monotone sequences \cite{IIBBC03} and of variable metric 
Fej\'er monotone sequences \cite{IICV13b}.

\begin{definition}
\label{IIdef:1}
Let $\XX$ be a reflexive real Banach space, 
let $f\in\Gamma_0(\XX)$ be G\^ateaux differentiable on 
$\IDD\neq\emp$, let $(f_n)_{n\in\NN}$ 
be in $\BF(f)$, let $(x_n)_{n\in\NN}\in(\IDD)^{\NN}$, 
and let $C\subset\XX$ be such that 
$C\cap\dom f\neq\emp$. Then $(x_n)_{n\in\mathbb{N}}$ is: 
\begin{enumerate} 
\item
\emph{quasi-Bregman monotone} with respect to $C$ relative to
$(f_n)_{n\in\mathbb{N}}$ if 
\begin{multline}
\label{IIe:qbm}
(\exists(\eta_n)_{n\in\mathbb{N}}\in\ell_+^1(\mathbb{N}))
(\forall x\in C\cap\dom f)
(\exists(\varepsilon_n)_{n\in\mathbb{N}}\in
\ell_+^1(\mathbb{N}))(\forall n\in\mathbb{N})\\ 
D^{f_{n+1}}(x,x_{n+1})\leq
(1+\eta_n)D^{f_n}(x,x_n)+\varepsilon_n;
\end{multline}
\item 
\emph{stationarily quasi-Bregman monotone} with respect to $C$
relative to $(f_n)_{n\in\mathbb{N}}$ if
\begin{multline}
\label{IIe:sqbm}
(\exists(\varepsilon_n)_{n\in\mathbb{N}}\in\ell_+^1(\mathbb{N}))
(\exists(\eta_n)_{n\in\mathbb{N}}\in\ell_+^1(\mathbb{N}))(\forall
x\in C\cap\dom f)(\forall n\in\mathbb{N})\\
D^{f_{n+1}}(x,x_{n+1})\leq (1+\eta_n)D^{f_n}(x,x_n)+\varepsilon_n.
\end{multline}
\end{enumerate}
\end{definition}

\begin{remark}\
\label{IIrm:1}
\begin{enumerate}
\item 
In Definition~\ref{IIdef:1}, suppose that $(\forall n\in\NN)$ 
$f_n=\hat{f}$ and $\eta_n=\varepsilon_n=0$. 
Then we recover the notion of a Bregman monotone sequence 
defined in \cite{IIBBC03}.
\item
\label{IIrm:i}
In Definition~\ref{IIdef:1}, suppose that $\XX$ is a Hilbert 
space, that $f=\|\cdot\|^2/2$, and that $(\forall n\in\NN)$ 
$f_n\colon x\mapsto\pair{x}{U_nx}/2$, where $(U_n)_{n\in\NN}$ 
are operators in $\BP_{\alpha}(\XX)$ for some 
$\alpha\in\left[0,+\infty\right[$. 
Then we recover \cite[Definition~2.1]{IICV13b} 
with $\phi=|\cdot|^2/2$. 
\end{enumerate}
\end{remark}

Here are some basic properties of quasi-Bregman monotone sequences.

\begin{proposition}
\label{IIpp:qb1}
Let $\XX$ be a reflexive real Banach space, 
let $f\in\Gamma_0(\XX)$ be G\^ateaux differentiable on 
$\IDD\neq\emp$, let $\alpha\in\RPP$, 
let $(f_n)_{n\in\NN}$ be in $\BP_{\alpha}(f)$, let 
$(x_n)_{n\in\NN}\in(\IDD)^{\NN}$, let 
$C\subset\XX$ be such that $C\cap\IDD\neq\emp$, 
and let $x\in C\cap\IDD$. Suppose that 
$(x_n)_{n\in\mathbb{N}}$ is quasi-Bregman monotone with 
respect to $C$ relative to $(f_n)_{n\in\mathbb{N}}$. 
Then the following hold:
\begin{enumerate}
\item
\label{IIpp:qb1i}
$(D^{f_n}(x,x_n))_{n\in\NN}$ converges.
\item
\label{IIpp:qb1ii}
Suppose that $D^f(x,\cdot)$ is coercive. Then 
$(x_n)_{n\in\mathbb{N}}$ is bounded.
\end{enumerate}
\end{proposition}

\begin{proof}
\ref{IIpp:qb1i}: 
Let us set $(\forall n\in\NN)$ $\xi_n=D^{f_n}(x,x_n)$. 
Since $(x_n)_{n\in\NN}$ is quasi-Bregman monotone with 
respect to $C$ relative to $(f_n)_{n\in\NN}$, 
there exist $(\eta_n)_{n\in\NN}\in\ell_{+}^1(\NN)$ and 
$(\varepsilon_n)_{n\in\NN}\in\ell_{+}^1(\NN)$ such that
\begin{equation}
\label{IIe:2b}
(\forall n\in\NN)\quad\xi_{n+1}\leq(1+\eta_n)\xi_n+\varepsilon_n.
\end{equation}
It therefore follows from \cite[Lemma~2.2.2]{IIPo87_B} 
that $(\xi_n)_{n\in\NN}$ converges, i.e., 
$(D^{f_n}(x,x_n))_{n\in\NN}$ converges.

\ref{IIpp:qb1ii}: 
Since $(f_n)_{n\in\NN}$ 
is in $\BP_{\alpha}(f)$, we deduce that 
\begin{equation}
(\forall n\in\NN)\quad D^f(x,x_n)\leq\alpha^{-1}D^{f_n}(x,x_n).
\end{equation}
Therefore, since \ref{IIpp:qb1i} implies that 
$(D^{f_n}(x,x_n))_{n\in\NN}$ is bounded, 
$(D^f(x,x_n))_{n\in\NN}$ is bounded. 
In turn, since $D^f(x,\cdot)$ is coercive, 
$(x_n)_{n\in\mathbb{N}}$ is bounded.
\end{proof}

The following result concerns the weak convergence of 
quasi-Bregman monotone sequences.

\begin{proposition}
\label{IIpp:qb2}
Let $\XX$ be a reflexive real Banach space, 
let $f\in\Gamma_0(\XX)$ be G\^ateaux differentiable 
on $\IDD\neq\emp$, let $(x_n)_{n\in\NN}\in(\IDD)^{\NN}$, 
let $C\subset\XX$ be such that $C\cap\IDD\neq\emp$, 
let $(\eta_n)_{n\in\NN}\in\ell_{+}^1(\NN)$, 
let $\alpha\in\RPP$, and let $(f_n)_{n\in\NN}$ in 
$\BP_{\alpha}(f)$ be such that $(\forall n\in\NN)$ 
$(1+\eta_n)f_n\succcurlyeq f_{n+1}$. Suppose that 
$(x_n)_{n\in\mathbb{N}}$ is quasi-Bregman 
monotone with respect to $C$ relative to 
$(f_n)_{n\in\mathbb{N}}$, that there exists $g\in\BF(f)$ 
such that for every $n\in\NN$, $g\succcurlyeq f_n$, 
and, for every $x_1\in\XX$ and every $x_2\in\XX$,
\begin{equation}
\label{IIe:4p}
\begin{cases}
x_1\in\mathfrak{W}(x_n)_{n\in\NN}\cap C\\
x_2\in\mathfrak{W}(x_n)_{n\in\NN}\cap C\\
\big(\pair{x_1-x_2}{\nabla f_n(x_n)}\big)_{n\in\NN}
\quad\text{converges}
\end{cases}
\Rightarrow\quad x_1=x_2.
\end{equation}
Moreover, suppose that $(\forall x\in\IDD)$ 
$D^f(x,\cdot)$ is coercive. Then $(x_n)_{n\in\NN}$ converges 
weakly to a point in $C\cap\IDD$ if and only if 
$\mathfrak{W}(x_n)_{n\in\NN}\subset C\cap\IDD$.
\end{proposition}

\begin{proof}
Necessity is clear. To show sufficiency, suppose that every weak 
sequential cluster point of $(x_n)_{n\in\NN}$ is in 
$C\cap\IDD$ and let $x_1$ and $x_2$ be two such points. 
First, it follows from Proposition~\ref{IIpp:qb1}\ref{IIpp:qb1i} 
that
\begin{equation}
\label{IIe:20p}
\big(D^{f_n}(x_1,x_n)\big)_{n\in\NN}
\quad\text{and}\quad
\big(D^{f_n}(x_2,x_n)\big)_{n\in\NN}
\quad\text{are convergent}.
\end{equation}
Next, let us define the following functions 
\begin{equation}
\phi\colon\left[0,1\right]\to\RR\colon t\mapsto
\Pair{x_1-x_2}{\nabla g(x_2+t(x_1-x_2))-\nabla g(x_2)},
\end{equation}
and
\begin{equation}
(\forall n\in\NN)\quad
\phi_n\colon\left[0,1\right]\to\RR\colon t\mapsto
\Pair{x_1-x_2}{\nabla f_n(x_2+t(x_1-x_2))
-\nabla f_n(x_2)}.
\end{equation}
Then
\begin{equation}
\int_0^1\phi(t)dt=g(x_1)-g(x_2)
\quad\text{and}\quad(\forall n\in\NN)\quad
\int_0^1\phi_n(t)dt=f_n(x_1)-f_n(x_2).
\end{equation}
For every $n\in\NN$, since 
$(1+\eta_n)f_n\succcurlyeq f_{n+1}$, 
for every $t\in\left]0,1\right])$, we have
\begin{align}
\phi_{n+1}(t)
&=\Pair{x_1-x_2}{\nabla f_{n+1}(x_2+t(x_1-x_2))
-\nabla f_{n+1}(x_2)}\nonumber\\
&=t^{-1}\Pair{x_2+t(x_1-x_2)-x_2}{\nabla f_{n+1}(x_2+t(x_1-x_2))
-\nabla f_{n+1}(x_2)}\nonumber\\
&=t^{-1}\big(D^{f_{n+1}}\big(x_2+t(x_1-x_2),x_2\big)
+D^{f_{n+1}}\big(x_2,x_2+t(x_1-x_2)\big)\big)\nonumber\\
&\leq (1+\eta_n)t^{-1}\big(D^{f_n}\big(x_2+t(x_1-x_2),x_2\big)
+D^{f_n}\big(x_2,x_2+t(x_1-x_2)\big)\big)\nonumber\\
&=(1+\eta_n)t^{-1}\Pair{x_2+t(x_1-x_2)-x_2}
{\nabla f_n(x_2+t(x_1-x_2))
-\nabla f_n(x_2)}\nonumber\\
&=(1+\eta_n)\Pair{x_1-x_2}
{\nabla f_n(x_2+t(x_1-x_2))
-\nabla f_n(x_2)}\nonumber\\
&=(1+\eta_n)\phi_n(t).
\end{align}
Consequently,
\begin{equation}
\label{IIe:21p}
(\forall n\in\NN)(\forall t\in\left]0,1\right])
\quad 0\leq\phi_{n+1}(t)\leq(1+\eta_n)\phi_n(t).
\end{equation}
It is clear that \eqref{IIe:21p} is valid for $t=0$ since 
in this case, all terms are equal to $0$. 
In turn, we deduce from \cite[Lemma~2.2.2]{IIPo87_B} that
\begin{equation}
\label{IIe:22p}
(\phi_n)_{n\in\NN}\quad\text{converges pointwise}.
\end{equation}
On the other hand, for every $n\in\NN$, 
since $g\succcurlyeq f_n$, the same argument 
as above shows that
\begin{equation}
\label{IIe:23p}
(\forall t\in\left[0,1\right])
\quad 0\leq\phi_n(t)\leq\phi(t).
\end{equation}
By invoking \eqref{IIe:22p}, \eqref{IIe:23p}, and Lebesgue's 
dominated convergence theorem, we obtain that
\begin{equation}
\Bigg(\int_0^1\phi_n(t)dt\Bigg)_{n\in\NN}\quad\text{converges,}
\end{equation}
which implies that 
\begin{equation}
\label{IIe:24p}
\Big(f_n(x_1)-f_n(x_2)\Big)_{n\in\NN}\quad\text{converges.}
\end{equation}
We also observe that 
\begin{equation}
(\forall n\in\NN)\quad 
D^{f_n}(x_1,x_n)-D^{f_n}(x_2,x_n)=f_n(x_1)-f_n(x_2)
-\pair{x_1-x_2}{\nabla f_n(x_n)},
\end{equation}
and hence, it follows from \eqref{IIe:20p} 
and \eqref{IIe:24p} that 
\begin{equation}
\label{IIIe:25}
\big(\pair{x_1-x_2}{\nabla f_n(x_n)}\big)_{n\in\NN}
\quad\text{converges}.
\end{equation}
In turn, \eqref{IIe:4p} forces $x_1=x_2$. 
Since Proposition~\ref{IIpp:qb1}\ref{IIpp:qb1ii} asserts 
that $(x_n)_{n\in\NN}$ is bounded and since $\XX$ is reflexive, 
we conclude that $x_n\weakly x_1\in C\cap\IDD$.
\end{proof}

\begin{example}
Let $\XX$ be a reflexive real Banach space, 
let $f\in\Gamma_0(\XX)$ be G\^ateaux differentiable 
on $\IDD\neq\emp$, let $(f_n)_{n\in\NN}$ 
be in $\BF(f)$, let $(x_n)_{n\in\NN}\in(\IDD)^{\NN}$, 
and let $C\subset\XX$. Suppose that 
$C\cap\overline{\dom}f$ is a singleton. 
Then \eqref{IIe:4p} is satisfied.
\end{example}
\begin{proof}
Since $(x_n)_{n\in\NN}\in(\IDD)^{\NN}$, 
$\mathfrak{W}(x_n)_{n\in\NN}\subset\overline{\dom}f$, 
and therefore, $\mathfrak{W}(x_n)_{n\in\NN}\cap C$ 
is at most a singleton.
\end{proof}

\begin{example}
Let $\XX$ be a reflexive real Banach space, 
let $f\in\Gamma_0(\XX)$ be G\^ateaux differentiable 
on $\IDD\neq\emp$, 
let $(x_n)_{n\in\NN}\in(\IDD)^{\NN}$, 
let $C\subset\IDD$, and set $(\forall n\in\NN)$ $f_n=\hat{f}$. 
Suppose that $f|_{\IDD}$ is strictly convex 
and that $\nabla f$ is weakly sequentially continuous. 
Then \eqref{IIe:4p} is satisfied.
\end{example}
\begin{proof}
Suppose that $x_1\in\mathfrak{W}(x_n)_{n\in\NN}\cap C$ 
and $x_2\in\mathfrak{W}(x_n)_{n\in\NN}\cap C$ 
are such that $(\pair{x_1-x_2}{\nabla f_n(x_n})_{n\in\NN}$ 
converges and $x_1\neq x_2$. Take strictly increasing sequences 
$(k_n)_{n\in\NN}$ and $(l_n)_{n\in\NN}$ in $\NN$ such that 
$x_{k_n}\weakly x_1$ and $x_{l_n}\weakly x_2$. 
Since $\nabla f$ is weakly sequentially continuous, 
by taking the limit in \eqref{IIe:4p} along subsequences 
$(x_{k_n})_{n\in\NN}$ and $(x_{l_n})_{n\in\NN}$, we get
\begin{equation}
\label{IIe:5r}
\pair{x_1-x_2}{\nabla f(x_1)-\nabla f(x_2)}=0
\end{equation}
Since $f|_{\IDD}$ is strictly convex, 
$\nabla f$ is strictly monotone 
\cite[Theorem~2.4.4(ii)]{IIZa02_B}, i.e.,
\begin{equation}
\pair{x_1-x_2}{\nabla f(x_1)-\nabla f(x_2)}>0,
\end{equation}
and we reach a contradiction.
\end{proof}

\begin{example}
Let $\XX$ be a real Hilbert space, 
let $f=\|\cdot\|^2/2$, let $C\subset\XX$, 
let $(x_n)_{n\in\NN}$ be a sequence in $\XX$, 
let $\alpha\in\RPP$, let $U$ and 
$(U_n)_{n\in\NN}$ be self-adjoint linear operators 
from $\XX$ in $\XX$ such that $U_n\to U$ pointwise, 
and set $(\forall n\in\NN)$ $f_n=\pair{\cdot}{U_n\cdot}/2$. 
Suppose that $\pair{\cdot}{U\cdot}\geq\alpha\|\cdot\|^2$. 
Then \eqref{IIe:4p} is satisfied.
\end{example}
\begin{proof}
It is easy to see that, for every $n\in\NN$, 
$f_n$ is G\^ateaux differentiable on $\XX$ with 
$\nabla f_n=U_n$. 
Suppose that $x_1\in\mathfrak{W}(x_n)_{n\in\NN}\cap C$ 
and $x_2\in\mathfrak{W}(x_n)_{n\in\NN}\cap C$ 
are such that $(\pair{x_1-x_2}{\nabla f_n(x_n})_{n\in\NN}$ 
converges. Take strictly increasing sequences 
$(k_n)_{n\in\NN}$ and $(l_n)_{n\in\NN}$ in $\NN$ 
such that $x_{k_n}\weakly x_1$ and $x_{l_n}\weakly x_2$. 
We have
\begin{equation}
\Pair{x_1-x_2}{\nabla f_{k_n}(x_{k_n})}
=\Pair{x_1-x_2}{U_{k_n}x_{k_n}}
=\Pair{U_{k_n}x_1-U_{k_n}x_2}{x_{k_n}}\to\Pair{Ux_1-Ux_2}{x_1},
\end{equation}
and
\begin{equation}
\Pair{x_1-x_2}{\nabla f_{l_n}(x_{l_n})}
=\Pair{x_1-x_2}{U_{l_n}x_{l_n}}
=\Pair{U_{l_n}x_1-U_{l_n}x_2}{x_{l_n}}\to\Pair{Ux_1-Ux_2}{x_2},
\end{equation}
and hence, $0=\Pair{Ux_1-Ux_2}{x_1-x_2}\geq\alpha\|x_1-x_2\|^2$, 
and therefore, $x_1=x_2$.
\end{proof}

The following condition will be used subsequently 
(see \cite[Examples~4.10, 5.11, and 5.13]{IIBBC03} 
for special cases).

\begin{condition}
{\rm\cite[Condition~4.4]{IIBBC03}}
\label{IIcd:2}
Let $\XX$ be a reflexive real Banach space 
and let $f\in\Gamma_0(\XX)$ be G\^ateaux differentiable on 
$\IDD\neq\emp$. For every bounded sequences 
$(x_n)_{n\in\NN}$ and $(y_n)_{n\in\NN}$ in $\IDD$,
\begin{equation}
\label{IIe:5}
D^f(x_n,y_n)\to 0\quad\Rightarrow\quad x_n-y_n\to 0.
\end{equation}
\end{condition}

We now present a characterization of the strong convergence 
of stationarily quasi-Bregman monotone sequences.

\begin{proposition}
\label{IIpp:qb3} 
Let $\XX$ be a reflexive real Banach space, 
let $f\in\Gamma_0(\XX)$ be a Legendre function, 
let $\alpha\in\RPP$, 
let $(f_n)_{n\in\NN}$ be in $\BP_{\alpha}(f)$, 
let $(x_n)_{n\in\NN}\in(\IDD)^{\NN}$, and let 
$C$ be a closed convex subset of $\XX$ such that 
$C\cap\IDD\neq\emp$. Suppose that 
$(x_n)_{n\in\mathbb{N}}$ is stationarily quasi 
Bregman monotone with respect to $C$ relative to 
$(f_n)_{n\in\mathbb{N}}$, that $f$ satisfies 
Condition~\ref{IIcd:2}, and that 
$(\forall x\in\IDD)$ $D^f(x,\cdot)$ is coercive. 
In addition, suppose that there exists $\beta\in\RPP$ 
such that $(\forall n\in\NN)$ 
$\beta\hat{f}\succcurlyeq f_n$. 
Then $(x_n)_{n\in\NN}$ converges 
strongly to a point in $C\cap\overline{\dom}f$ 
if and only if $\varliminf D^f_C(x_n)=0$. 
\end{proposition}

\begin{proof}
To show the necessity, suppose that 
$x_n\to\overline{x}\in C\cap\overline{\dom}f$ 
and take $x\in C\cap\IDD$. 
Since Proposition~\ref{IIpp:qb1}\ref{IIpp:qb1i} 
states that $(D^{f_n}(x,x_n))_{n\in\NN}$ is bounded 
and since
\begin{equation}
(\forall n\in\NN)\quad D^f(x,x_n)\leq D^{f_n}(x,x_n),
\end{equation}
we deduce that $(D^f(x,x_n))_{n\in\NN}$ is bounded. 
However, by \cite[Lemma~7.3(vii)]{IIBBC01},
\begin{equation}
\label{IIe:bounded1}
(\forall n\in\NN)\quad
D^{f^*}\big(\nabla f(x_n),\nabla f(x)\big)=D^f(x,x_n).
\end{equation}
Therefore $(D^{f^*}(\nabla f(x_n),\nabla f(x)))_{n\in\NN}$ 
is bounded. In turn, since $D^{f^*}(\cdot,\nabla f(x))$ 
is coercive \cite[Lemma~7.3(v)]{IIBBC01}, 
we get $(\nabla f(x_n))_{n\in\NN}$ is bounded 
and hence $\pair{\overline{x}-x_n}{\nabla f(x_n)}\to 0$. 
Since
\begin{align}
(\forall n\in\NN)\quad 
D_C^f(x_n)&=\inf D^f(C,x_n)\nonumber\\
&\leq\inf D^f(C\cap\overline{\dom}f,x_n)\nonumber\\
&\leq D^f(\overline{x},x_n)\nonumber\\
&=f(\overline{x})-f(x_n)-\Pair{\overline{x}-x_n}{\nabla f(x_n)},
\end{align}
we obtain
\begin{equation}
\label{IIe:bounded2}
\varliminf D_C^f(x_n)\leq
f(\overline{x})-\varlimsup f(x_n)
-\lim\Pair{\overline{x}-x_n}{\nabla f(x_n)}
=f(\overline{x})-\varlimsup f(x_n).
\end{equation}
Since $f$ is lower semicontinuous,
\begin{equation}
\label{IIe:bounded3}
f(\overline{x})\leq\varliminf f(x_n)\leq\varlimsup f(x_n).
\end{equation}
Altogether, \eqref{IIe:bounded2} and \eqref{IIe:bounded3} yield
\begin{equation}
\varliminf D_C^f(x_n)\to 0.
\end{equation}
We now show the sufficiency. 
First, since $f$ is Legendre and $C\cap\IDD\neq\emp$, 
\eqref{IIe:2001} yields
\begin{equation}
\label{IIe:5b}
P_C^f\colon\IDD\to C\cap\IDD.
\end{equation}
Next, we set
\begin{equation}
(\forall n\in\NN)\quad\varrho_n=D^f_C(x_n)
\quad\text{and}\quad\zeta_n
=\inf_{x\in C\cap\dom f}D^{f_n}(x,x_n).
\end{equation}
Then $\varliminf\varrho_n=0$. For every $n\in\NN$, 
since $\beta\hat{f}\succcurlyeq f_n\succcurlyeq\alpha f$, 
we obtain
\begin{equation}
(\forall x\in C\cap\dom f)\quad 
0\leq\alpha D^f(x,x_n)\leq D^{f_n}(x,x_n)\leq\beta D^f(x,x_n).
\end{equation}
In the above inequalities, after taking the infimum over 
$x\in C\cap\dom f$, we get
\begin{equation}
\label{IIe:8}
(\forall n\in\NN)\quad 0\leq\alpha\varrho_n
\leq\zeta_n\leq\beta\varrho_n
\end{equation}
and therefore,
\begin{equation}
\label{IIe:10}
0\leq\alpha\varliminf\varrho_n\leq\varliminf\zeta_n
\leq\beta\varliminf\varrho_n=0.
\end{equation}
On the other hand, since 
$(x_n)_{n\in\mathbb{N}}$ is stationarily quasi 
Bregman monotone with respect to $C$ relative to 
$(f_n)_{n\in\mathbb{N}}$, there exist 
$(\eta_n)_{n\in\NN}\in\ell_{+}^1(\NN)$ and 
$(\varepsilon_n)_{n\in\NN}\in\ell_{+}^1(\NN)$ such that
\begin{equation}
\label{IIe:6}
(\forall x\in C\cap\dom f)(\forall n\in\NN)\quad 
D^{f_{n+1}}(x,x_{n+1})
\leq(1+\eta_n)D^{f_n}(x,x_n)+\varepsilon_n.
\end{equation}
Taking the infimum in \eqref{IIe:6} over $C\cap\dom f$ yields
\begin{equation}
\label{IIe:7}
(\forall n\in\NN)\quad\zeta_{n+1}\leq(1+\eta_n)\zeta_n
+\varepsilon_n. 
\end{equation}
It therefore follows from \cite[Lemma~2.2.2]{IIPo87_B} 
that $(\zeta_n)_{n\in\NN}$ converges, and thus, 
we deduce from \eqref{IIe:10} that $\zeta_n\to 0$. 
Appealing to \eqref{IIe:8}, we get $\varrho_n\to 0$, i.e.,
\begin{equation}
\label{IIe:11}
D^f\big(P_{C}^fx_n,x_n\big)\to 0.
\end{equation} 
Now let $x\in C\cap\IDD$. Then $x\in\Fix P_{C}^f$ 
\cite[Proposition~3.22(ii)(b)]{IIBBC03} and it follows from 
Proposition~\ref{IIpp:qb1}\ref{IIpp:qb1i} that 
$(D^{f_n}(x,x_n))_{n\in\NN}$ is bounded, 
and hence, $(D^f(x,x_n))_{n\in\NN}$ is likewise. In turn, 
since \cite[Proposition~3.3(i) and Theorem~3.34]{IIBBC03} 
yield  
\begin{equation}
\label{IIe:12}
(\forall n\in\NN)\quad D^f\big(x,P^f_Cx_n\big)\leq D^f(x,x_n),
\end{equation}
we deduce that $(D^f(x,P^f_Cx_n))_{n\in\NN}$ is bounded, 
and hence, since $D^f(x,\cdot)$ is coercive, we obtain that
\begin{equation}
\label{IIe:13}
\big(P^f_Cx_n\big)_{n\in\NN}\in(\IDD)^{\NN}
\;\;\text{is bounded}.
\end{equation}
Therefore, since $f$ satisfies Condition~\ref{IIcd:2}, 
it follows from \eqref{IIe:11} that
\begin{equation}
\label{IIe:14}
P_{C}^fx_n-x_n\to 0.
\end{equation}
Since \eqref{IIe:5b} entails that
\begin{equation}
\label{IIe:14b}
(\forall n\in\NN)\quad P_C^fx_n\in C\cap\IDD=\Fix P_C^f, 
\end{equation}
we obtain
\begin{equation}
\label{IIe:15}
(\forall n\in\NN)\quad 
0\leq d_C(x_n)=\inf\limits_{x\in C}\|x-x_n\|\leq\|P_C^fx_n-x_n\|.
\end{equation}
Altogether, \eqref{IIe:14} and \eqref{IIe:15} imply 
that
\begin{equation}
\label{IIe:16}
d_C(x_n)\to 0.
\end{equation}
Set $\tau=\prod_{k\in\NN}(1+\eta_k)$. Then $\tau<+\infty$ 
\cite[Theorem~3.7.3]{IIKno56_B}. By invoking \eqref{IIe:14b} 
and \cite[Proposition~3.3(i) and Theorem~3.34]{IIBBC03}, we get
\begin{align}
\label{IIe:17}
(\forall n\in\NN)(\forall m\in\NN)\quad D^f\big(P_{C}^fx_n,
P_{C}^fx_{m+n}\big)&\leq D^f\big(P_{C}^fx_n,x_{m+n}\big)
\nonumber\\
&\leq\alpha^{-1}D^{f_{m+n}}\big(P_{C}^fx_n,x_{m+n}\big)
\nonumber\\
&\leq\tau\alpha^{-1}\Bigg(D^{f_n}\big(P_{C}^fx_n,x_n\big)
+\sum_{k=n}^{n+m-1}\varepsilon_k\Bigg)\nonumber\\
&\leq\tau\alpha^{-1}\Bigg(\beta D^f\big(P_{C}^fx_n,x_n\big)
+\sum_{k\geq n}\varepsilon_k\Bigg)\nonumber\\
&=\tau\alpha^{-1}\Bigg(\beta\varrho_n+\sum_{k\geq n}
\varepsilon_k\Bigg).
\end{align}
After taking the limit as $n\to +\infty$ and $m\to +\infty$ 
in \eqref{IIe:17}, we obtain
\begin{equation}
D^f\big(P_{C}^fx_{m+n},P_{C}^fx_n\big)\to 0,
\end{equation}
and thus \eqref{IIe:13} yield
\begin{equation}
\label{IIe:18}
P_{C}^fx_{m+n}-P_{C}^fx_n\to 0.
\end{equation}
However,
\begin{equation}
\label{IIe:19}
(\forall n\in\NN)(\forall m\in\NN)\quad\|x_{m+n}-x_n\|\leq
\|x_{m+n}-P_{C}^fx_{m+n}\|+\|P_{C}^fx_{m+n}-P_{C}^fx_n\|+
\|P_{C}^fx_n-x_n\|.
\end{equation}
After taking the limit as $n\to +\infty$ and $m\to +\infty$ 
in \eqref{IIe:19} then using \eqref{IIe:14} and 
\eqref{IIe:18}, we get
\begin{equation}
\|x_{n+m}-x_n\|\to 0.
\end{equation}
Thus, $(x_n)_{n\in\NN}$ is a Cauchy sequence in $\XX$, 
and hence, there exists $\overline{x}\in\XX$ such that 
$x_n\to\overline{x}$. By \eqref{IIe:16} and the 
continuity of $d_C$ \cite[Example~1.47]{IIBC11_B}, we obtain 
$d_C(\overline{x})=0$ and, since $C$ is closed, 
$\overline{x}\in C$. Because $(x_n)_{n\in\NN}$ is in $\IDD$, 
we conclude that $\overline{x}\in\overline{\dom}f$.
\end{proof}

\begin{remark}
In Proposition~\ref{IIpp:qb3}, suppose that $\XX$ is a Hilbert 
space, that $f=\|\cdot\|^2/2$, and that $(\forall n\in\NN)$ 
$f_n\colon x\mapsto\pair{x}{U_nx}/2$, where $(U_n)_{n\in\NN}$ 
are operators in $\BP_{\alpha}(\XX)$ such that 
$\sup_{n\in\NN}\|U_n\|<\pinf$. Then we recover 
\cite[Theorem~3.4]{IICV13b} with 
$\phi=|\cdot|^2/2$.
\end{remark}

\section{Bregman distance-based proximity operators}
\label{IIsec:VBPA}
\noindent 
Many algorithms in optimization in a real Hilbert space 
$\HH$ are based on Moreau's proximity operator 
\cite{IIMor62b} of a function 
$\varphi\in\Gamma_0(\HH)$
\begin{equation}
\label{IIe:prox}
\prox_{\varphi}\colon\HH\to\HH\colon 
x\mapsto\argmin\big(\varphi+\|\cdot-x\|^2/2\big).
\end{equation}
Because the quadratic term in \eqref{IIe:prox} 
is difficult to manipulate in Banach spaces 
since its gradient is nonlinear, alternative notions based on 
Bregman distances have been used (see \cite{IIBBC03} 
and the references therein). 
This leads to the notion of $D^f$-proximal operators. 
In this section, we investigate some their basic properties. 

\begin{lemma}{\rm\cite[Section~3]{IIBBC03}}
\label{IIl:prox}
Let $\XX$ be a reflexive real Banach space, 
let $\varphi\in\Gamma_0(\XX)$ be bounded from below, 
and let $f\in\Gamma_0(\XX)$ be a Legendre function 
such that $\dom\varphi\cap\IDD\neq\emp$. 
Then the following hold:
\begin{enumerate}
\item
$\prox_{\varphi}^f$ is single-valued on its domain.
\item
$\ran\prox_{\varphi}^f\subset\dom\prox_{\varphi}^f=\IDD$.
\item
\label{IIl:proxiii}
$\prox_{\varphi}^f=(\nabla f+\partial\varphi)^{-1}\circ\nabla f$.
\item
$\Fix\prox_{\varphi}^f=\Argmin\varphi\cap\IDD$.
\item
\label{IIl:proxiv} 
Let $x\in\Argmin\varphi\cap\IDD$, 
let $y\in\IDD$, and let $v=\prox_{\varphi}^fy$. Then
\begin{equation}
D^f(x,v)+D^f(v,y)\leq D^f(x,y).
\end{equation}
\end{enumerate}
\end{lemma}

The following result in an extension of 
\cite[Proposition~23.30]{IIBC11_B}.

\begin{proposition}
\label{IIl:proxprod}
Let $m$ be a strictly positive integer, let 
$(\mathcal{X}_i)_{1\leq i\leq m}$ be reflexive real 
Banach spaces, and let $\XX$ be the vector 
product space $\Cart_{\!\!i=1}^{\!\!m}\XX_i$ equipped 
with the norm $x=(x_i)_{1\leq i\leq m}
\mapsto\sqrt{\sum_{i=1}^m\|x_i\|^2}$. 
For every $i\in\{1,\ldots,m\}$, 
let $\varphi_i\in\Gamma_0(\XX_i)$ be bounded from below 
and let $f_i\in\Gamma_0(\XX_i)$ be a Legendre function 
such that $\dom\varphi_i\cap\IDD_i\neq\emp$. 
Set $f\colon\XX\to\RX
\colon x\mapsto\sum_{i=1}^mf_i(x_i)$ 
and $\varphi\colon\XX\to\RX
\colon x\mapsto\sum_{i=1}^m\varphi_i(x_i)$. 
Then
\begin{equation}
\Big(\forall x\in
\underset{i=1}{\overset{m}{\Cart}}\IDD_i\Big)\quad 
\prox_{\varphi}^{f}x
=\big(\prox_{\varphi_i}^{f_i}x_i\big)_{1\leq i\leq m}.
\end{equation}
\end{proposition}
\begin{proof}
First, we observe that $\XX^*$ is the vector 
product space $\Cart_{\!\!i=1}^{\!\!m}\XX_i^*$ equipped 
with the norm $x^*=(x_i^*)_{1\leq i\leq m}
\mapsto\sqrt{\sum_{i=1}^m\|x_i^*\|^2}$. 
Since, for every $i\in\{1,\ldots,m\}$, $\varphi_i$ is 
bounded from below, so is $\varphi$. 
Next, we derive from the definition of $f$ 
that $\dom f=\Cart_{\!\!i=1}^{\!\!m}\dom f_i$ 
and that
\begin{equation}
\partial f\colon\XX\to 
2^{\XX^*}\colon (x_i)_{1\leq i\leq m}
\mapsto\underset{i=1}{\overset{m}{\Cart}}\partial f_i(x_i).
\end{equation}
Thus, $\partial f$ is single-valued on 
\begin{equation}
\dom\partial f
=\underset{i=1}{\overset{m}{\Cart}}\dom\partial f_i
=\underset{i=1}{\overset{m}{\Cart}}\IDD_i
=\inte\Big(\underset{i=1}{\overset{m}{\Cart}}\dom f_i\Big)
=\IDD.
\end{equation}
Likewise, since 
\begin{equation}
f^*\colon\XX^*\to\RX\colon 
(x_i^*)_{1\leq i\leq m}\mapsto\sum_{i=1}^mf_i^*(x_i^*), 
\end{equation}
we deduce that $\partial f^*$ is single-valued on 
$\dom\partial f^*=\IDDS$. 
Consequently, \cite[Theorems~5.4 and~5.6]{IIBBC01} assert that 
$f$ is a Legendre function. In addition,
\begin{equation}
\dom\varphi\cap\IDD
=\Big(\underset{i=1}{\overset{m}{\Cart}}\dom\varphi_i\Big)
\cap\Big(\underset{i=1}{\overset{m}{\Cart}}\IDD_i\Big)
=\underset{i=1}{\overset{m}{\Cart}}(\dom\varphi_i
\cap\IDD_i)\neq\emp.
\end{equation}
Now Lemma~\ref{IIl:prox} asserts that 
$\prox_{\varphi}^{f}\colon\IDD\to\dom\varphi\cap\IDD$. 
For the remainder of the proof, let 
$x\in\IDD$, set $p=\prox^{f}_{\varphi}x$, and set 
$q=(\prox^{f_i}_{\varphi_i}x_i)_{1\leq i\leq m}$. 
Since Lemma~\ref{IIl:prox}\ref{IIl:proxiii} 
yields $\nabla f(x)-\nabla f(p)\in\partial\varphi(p)$, 
we deduce from \eqref{IIe:subdiff} that
\begin{equation}
\label{IIe:23}
(\forall z\in\dom\varphi)\quad
\pair{z-p}{\nabla f(x)-\nabla f(p)}
+\varphi(p)
\leq\varphi(z).
\end{equation}
Setting $z=q$ in \eqref{IIe:23} yields
\begin{equation}
\label{IIe:24}
\pair{q-p}{\nabla f(x)-\nabla f(p)}
+\varphi(p)\leq
\varphi(q).
\end{equation}
For every $i\in\{1,\ldots,m\}$, set 
$q_i=\prox_{\varphi_i}^{f_i}x_i$. The same characterization 
as in \eqref{IIe:23} yields
\begin{equation}
\label{IIe:25}
(\forall i\in\{1,\ldots,m\})(\forall z_i\in\dom\varphi_i)
\quad\pair{z_i-q_i}{\nabla f_i(x_i)
-\nabla f_i(q_i)} +\varphi_i(q_i)\leq\varphi_i(z_i).
\end{equation}
By summing these inequalities over $i\in\{1,\ldots,m\}$, we obtain
\begin{equation}
\label{IIe:26}
(\forall z\in\dom\varphi)\quad
\pair{z-q}{\nabla f(x)-\nabla f(q)}
+\varphi(q)\leq\varphi(z).
\end{equation}
Upon setting $z=p$ in \eqref{IIe:26}, 
we get
\begin{equation}
\label{IIe:27}
\pair{p-q}{\nabla f(x)-\nabla f(q)}+\varphi(q)
\leq\varphi(p).
\end{equation}
Adding \eqref{IIe:24} and \eqref{IIe:27} yields
\begin{equation}
\label{IIe:28}
\pair{p-q}{\nabla f(p)-\nabla f(q)}\leq 0.
\end{equation}
Suppose that $p\neq q$. 
Since $f$ is essentially strictly convex, 
$f$ is strictly convex on every convex subset 
of $\dom\partial f$. In particular, since 
$\IDD\subset\dom\partial f$, 
$f|_{\IDD}$ is strictly convex. 
Hence, by \cite[Theorem~2.4.4(ii)]{IIZa02_B}, 
$\nabla f$ is strictly monotone, i.e.,
\begin{equation}
\label{IIe:29}
\pair{p-q}{\nabla f(p)-\nabla f(q)}>0,
\end{equation}
and we reach a contradiction. Consequently, 
$p=q$ which proves the claim.
\end{proof}

Let us note that, even in Euclidean spaces, it may be easier to 
evaluate $\prox_{\varphi}^f$ than Moreau's usual 
proximity operator $\prox_{\varphi}$, which is
based on $f=\|\cdot\|^2/2$. We provide illustrations of such 
instances in the standard Euclidean space $\RR^m$.

\begin{example}
\label{IIex:1}
Let $\gamma\in\RPP$, let $\phi\in\Gamma_0(\RR)$ be such that 
$\dom\phi\cap\RPP\neq\emp$, and let $\vartheta$ 
be Boltzmann-Shannon entropy, i.e.,
\begin{equation}
\vartheta\colon\xi\mapsto
\begin{cases}
\xi\ln\xi-\xi,&\text{if}\;\;\xi\in\RPP;\\
0,&\text{if}\;\;\xi=0;\\
\pinf,&\text{otherwise}.
\end{cases}
\end{equation}
Set $\varphi\colon(\xi_i)_{1\leq i\leq m}\mapsto
\sum_{i=1}^m\phi(\xi_i)$ and $f\colon(\xi_i)_{1\leq i\leq m}
\mapsto\sum_{i=1}^m\vartheta(\xi_i)$. 
Note that $f$ is a Legendre function 
\cite[Theorem~5.12 and Example~6.5]{IIBB97} and hence, 
Lemma~\ref{IIl:prox} asserts that 
$\dom\prox_{\gamma\varphi}^f=\RPP^m$. 
Let $(\xi_i)_{1\leq i\leq m}\in\RPP^m$, set 
$(\eta_i)_{1\leq i\leq m}
=\prox_{\gamma\varphi}^f(\xi_i)_{1\leq i\leq m}$, 
let $W$ be the Lambert function \cite{IIlamb96}, 
i.e., the inverse of $\xi\mapsto\xi e^{\xi}$ on $\RP$, 
and let $i\in\{1,\ldots,m\}$. 
Then $\eta_i$ can be computed as follows.
\begin{enumerate}
\item
\label{IIex:1ii}
Let $\omega\in\RR$ and suppose that 
\begin{equation}
\phi\colon\xi\mapsto
\begin{cases}
\xi\ln\xi-\omega\xi,&\text{if}\;\;\xi\in\RPP;\\
0,&\text{if}\;\;\xi=0;\\
+\infty,&\text{otherwise}.
\end{cases}
\end{equation}
Then $\eta_i=\xi_i^{(\omega-1)/(\gamma+1)}$.
\item 
\label{IIex:1iii}
Let $p\in\left[1,+\infty\right[$ 
and suppose that either $\phi=|\cdot|^p/p$ or
\begin{equation}
\phi\colon\xi\mapsto
\begin{cases}
\xi^p/p,&\text{if}\;\;\xi\in\RP;\\
+\infty,&\text{otherwise}.
\end{cases}
\end{equation}
Then 
\begin{equation}
\eta_i=
\begin{cases}
\left(\dfrac{W(\gamma(p-1)\xi_i^{p-1})}{\gamma(p-1)}
\right)^{\frac{1}{p-1}},&\text{if}\;\;p\in
\left]1,+\infty\right[;\\[4mm]
\xi_ie^{-\gamma},&\text{if}\;\;p=1.
\end{cases}
\end{equation}
\item
\label{IIex:1v}
Let $p\in\left[1,+\infty\right[$ and suppose that
\begin{equation}
\phi\colon\xi\mapsto
\begin{cases}
\xi^{-p}/p,&\text{if}\;\;\xi\in\RPP;\\
+\infty,&\text{otherwise}.
\end{cases}
\end{equation}
Then 
\begin{equation}
\eta_i=\left(\frac{W(\gamma(p+1)\xi_i^{-p-1})}{\gamma(p+1)}
\right)^{\frac{-1}{p+1}}.
\end{equation}
\item
\label{IIex:1vi}
Let $p\in\left]0,1\right[$ and suppose that
\begin{equation}
\phi\colon\xi\mapsto
\begin{cases}
-\xi^p/p,&\text{if}\;\;\xi\in\RP;\\
+\infty,&\text{otherwise}.
\end{cases}
\end{equation} 
Then 
\begin{equation}
\eta_i=\Bigg(\frac{W(\gamma(1-p)\xi_i^{p-1})}{\gamma(1-p)}
\Bigg)^{\frac{1}{p-1}}.
\end{equation}
\end{enumerate}
\end{example}

\begin{example}
\label{IIex:2}
Let $\phi\in\Gamma_0(\RR)$ be such that 
$\dom\phi\cap\left]0,1\right[\neq\emp$ 
and let $\vartheta$ be Fermi-Dirac entropy, i.e.,
\begin{equation}
\vartheta\colon\xi\mapsto
\begin{cases}
\xi\ln\xi-(1-\xi)\ln(1-\xi),&\text{if}\;\;\xi\in\left]0,1\right[;\\
0&\text{if}\;\;\xi\in\{0,1\};\\
\pinf,&\text{otherwise}.
\end{cases}
\end{equation}
Set $\varphi\colon(\xi_i)_{1\leq i\leq m}
\mapsto\sum_{i=1}^m\phi(\xi_i)$ 
and $f\colon(\xi_i)_{1\leq i\leq m}
\mapsto\sum_{i=1}^m\vartheta(\xi_i)$. 
Note that $f$ is a Legendre function 
\cite[Theorem~5.12 and Example~6.5]{IIBB97} and hence, 
Lemma~\ref{IIl:prox} asserts that 
$\dom\prox_{\varphi}^f=\left]0,1\right[^m$. 
Let $(\xi_i)_{1\leq i\leq m}\in\left]0,1\right[^m$, 
set $(\eta_i)_{1\leq i\leq m}
=\prox_{\varphi}^f(\xi_i)_{1\leq i\leq m}$, 
and let $i\in\{1,\ldots,m\}$. 
Then $\eta_i$ can be computed as follows.
\begin{enumerate}
\item
\label{IIex:2i}
Let $\omega\in\RR$ and suppose that
\begin{equation}
\phi\colon\xi\mapsto
\begin{cases}
\xi\ln\xi-\omega\xi,&\text{if}\;\;\xi\in\RPP;\\
0,&\text{if}\;\;\xi=0;\\
\pinf,&\text{otherwise}.
\end{cases}
\end{equation}
Then $\eta_i=e^{\omega}(2-2\xi_i)^{-1}
(-\xi_i+\sqrt{4\xi_i-3\xi_i^2})$.

\item
\label{IIex:2ii}
Suppose that
\begin{equation}
\phi\colon\xi\mapsto
\begin{cases}
(1-\xi)\ln(1-\xi)+\xi,&\text{if}\;\;\xi\in\left]-\infty,1\right[;\\
1&\text{if}\;\;\xi=1;\\
\pinf,&\text{otherwise}.
\end{cases}
\end{equation}
Then $\eta_i=1/2+\xi_i^{-1}/2-\sqrt{\xi_i^{-2}/4
+\xi_i^{-1}/2-3/4}$.
\end{enumerate}
\end{example}

\begin{example}
\label{IIex:3}
Let $\phi\in\Gamma_0(\RR)$ be such that 
$\dom\phi\cap\RPP\neq\emp$ and let $\vartheta$ 
be Burg entropy, i.e.,
\begin{equation}
\vartheta\colon\xi\mapsto
\begin{cases}
-\ln\xi,&\text{if}\;\xi\in\RPP;\\
\pinf,&\text{otherwise}.
\end{cases}
\end{equation}
Set $\varphi\colon(\xi_i)_{1\leq i\leq m}\mapsto
\sum_{i=1}^m\phi(\xi_i)$ and $f\colon(\xi_i)_{1\leq i\leq m}
\mapsto\sum_{i=1}^m\vartheta(\xi_i)$. 
Note that $f$ is a Legendre function 
\cite[Theorem~5.12 and Example~6.5]{IIBB97} and hence, 
Lemma~\ref{IIl:prox} asserts that 
$\dom\prox_{\varphi}^f=\RPP^m$. 
Let $(\xi_i)_{1\leq i\leq m}\in\RPP^m$, 
set $(\eta_i)_{1\leq i\leq m}
=\prox_{\varphi}^f(\xi_i)_{1\leq i\leq m}$, 
and let $i\in\{1,\ldots,m\}$. 
Then $\eta_i$ can be computed 
as follows.
\begin{enumerate}
\item
\label{IIex:3i}
Let $\gamma\in\RPP$ and suppose that $\phi=\gamma\vartheta$. 
Then $\eta_i=(1+\gamma)\xi_i$.

\item
\label{IIex:3ii}
Let $(\gamma,\alpha)\in\RP^2$, let $\omega\in\RR$, 
and suppose that
\begin{equation}
\phi\colon\xi\mapsto
\begin{cases}
-\gamma\ln\xi+\omega\xi+\alpha\xi^{-1},&\text{if}
\quad\xi\in\RPP;\\
\pinf,&\text{otherwise}.
\end{cases}
\end{equation}
Then $\eta_i=(2+2\omega\xi_i)^{-1}((\gamma+1)\xi_i
+\sqrt{(\gamma+1)^2\xi_i+4\alpha\xi_i(1+\omega\xi_i)})$.

\item 
\label{IIex:3iii}
Let $(\gamma,\alpha)\in\RP^2$, let $p\in\left[1,+\infty\right[$, 
and suppose that
\begin{equation}
\phi\colon\xi\mapsto
\begin{cases}
-\gamma\ln\xi+\alpha\xi^p,&\text{if}
\quad\xi\in\RPP;\\
\pinf,&\text{otherwise}.
\end{cases}
\end{equation}
Then $\eta_i$ is the strictly positive solution of 
$p\alpha\xi_i\eta^p+\rho=(\gamma+1)\xi_i$.

\item
\label{IIex:3iv}
Let $\alpha\in\RP$, let $p\in\left[1,+\infty\right[$, 
and suppose that
\begin{equation}
\phi\colon\xi\mapsto
\begin{cases}
\alpha\xi^{-p},&\text{if}
\quad\xi\in\RPP;\\
\pinf,&\text{otherwise}.
\end{cases}
\end{equation}
Then $\eta_i$ is the strictly positive solution of 
$p\eta^{p+1}-\xi_i\eta^p=\alpha p\xi_i$.
\end{enumerate}
\end{example}

\begin{example}
\label{IIex:4}
Let $f\colon(\xi_i)_{1\leq i\leq m}\mapsto\sum_{i=1}^m
\vartheta(\xi_i)$, 
where $\vartheta$ is Hellinger-like function, i.e.,
\begin{equation}
\vartheta\colon\xi\mapsto
\begin{cases}
-\sqrt{1-\xi^2},&\text{if}\;\;
\xi\in\left[-1,1\right];\\
\pinf,&\text{otherwise},
\end{cases}
\end{equation}
let $\gamma\in\RPP$, and let $\varphi=f$. 
Note that $f$ is a Legendre function 
\cite[Theorem~5.12 and Example~6.5]{IIBB97} and hence, 
Lemma~\ref{IIl:prox} asserts that 
$\dom\prox_{\gamma\varphi}^f=\left]-1,1\right[^m$. 
Let $(\xi_i)_{1\leq i\leq m}\in\left]-1,1\right[^m$ 
and set $(\eta_i)_{1\leq i\leq m}
=\prox_{\gamma\varphi}^f(\xi_i)_{1\leq i\leq m}$. 
Then $(\forall i\in\{1,\ldots,m\})$ 
$\eta_i=\xi_i/\sqrt{(\gamma+1)^2
+(\gamma^2+2\gamma+2)\xi_i^2}$.
\end{example}

\section{Applications}
\label{IIsec:app}

\subsection{Variable Bregman proximal point algorithm}
\noindent
The convex minimization problem, i.e., 
the problem of minimizing a convex function, 
can be solved by proximal point algorithm 
(see \cite{IIBC11_B,IICV13b} for Hilbertian setting 
and \cite{IIBBC03} for Banach space setting). 
In this section, we develop a proximal point algorithm 
which employs different Bregman distances at each iteration. 
This provides a unified framework 
for existing proximal point algorithms.

\begin{theorem}
\label{IIt:1}
Let $\XX$ be a reflexive real Banach space, 
let $\varphi\in\Gamma_0(\XX)$, let $f\in\Gamma_0(\XX)$ 
be a Legendre function such that $\Argmin\varphi\cap\IDD\neq\emp$, 
let $(\eta_n)_{n\in\NN}\in\ell_{+}^1(\NN)$, 
let $\alpha\in\RPP$, and let $(f_n)_{n\in\NN}$ 
be Legendre functions in $\BP_{\alpha}(f)$ such that 
\begin{equation}
\label{IIe:36a}
(\forall n\in\NN)\quad(1+\eta_n)f_n\succcurlyeq f_{n+1}.
\end{equation}
Let $x_0\in\IDD$, let $(\gamma_n)_{n\in\NN}\in\RPP^{\NN}$ 
be such that $\gamma=\inf_{n\in\NN}\gamma_n>0$, and iterate
\begin{equation}
\label{IIal:1}
(\forall n\in\NN)\quad x_{n+1}=\prox_{\gamma_n\varphi}^{f_n}x_n.
\end{equation}
Then the following hold:
\begin{enumerate}
\item
\label{IIt:1i}
$(x_n)_{n\in\NN}$ is stationarily Bregman monotone with 
respect to $\Argmin\varphi$ relative to $(f_n)_{n\in\NN}$.
\item
\label{IIt:1ii} 
$(x_n)_{n\in\NN}$ is a minimizing sequence of $\varphi$.
\item
\label{IIt:1iii}
Suppose that, for every $x\in\IDD$, 
$D^f(x,\cdot)$ is coercive, and that one of the following holds:
\begin{enumerate}
\item
\label{IIt:1iiia}
$\Argmin\varphi\cap\overline{\dom}f$ is a singleton.
\item
\label{IIt:1iiib}
Either $\Argmin\varphi\subset\IDD$ 
or $\dom f^*$ is open and $\nabla f^*$ 
is weakly sequentially continuous, 
there exists $g\in\BF(f)$ 
such that, for every $n\in\NN$, $g\succcurlyeq f_n$, 
and, for every $x_1\in\XX$ and every $x_2\in\XX$,
\begin{equation}
\label{IIe:4r}
\begin{cases}
x_1\in\mathfrak{W}(x_n)_{n\in\NN}\\
x_2\in\mathfrak{W}(x_n)_{n\in\NN}\\
\big(\Pair{x_1-x_2}{\nabla f_n(x_n)}\big)_{n\in\NN}
\quad\text{converges}
\end{cases}
\Rightarrow\quad x_1=x_2.
\end{equation}
\end{enumerate}
Then there exists $\overline{x}\in\Argmin\varphi$ 
such that $x_n\weakly\overline{x}$.
\item
\label{IIt:1iv}
Suppose that that $f$ satisfies Condition~\ref{IIcd:2} 
and that $(\forall x\in\IDD)$ $D^f(x,\cdot)$ is coercive. 
Furthermore, assume that $\varliminf D^f_{\Argmin\varphi}(x_n)=0$ 
and that there exists $\beta\in\RPP$ such that 
$(\forall n\in\NN)$ $\beta\hat{f}\succcurlyeq f_n$. 
Then there exists $\overline{x}\in\Argmin\varphi$ 
such that $x_n\to\overline{x}$.
\end{enumerate}
\end{theorem}
\begin{proof}
First, for every $n\in\NN$, since 
$\emp\neq\Argmin\varphi\cap\IDD\subset\dom\varphi\cap\IDD
=\dom\varphi\cap\IDD_n$, Lemma~\ref{IIl:prox} asserts that 
\begin{equation}
\label{IIe:36b}
\prox_{\gamma_n\varphi}^{f_n}\colon\IDD_n\to
\dom\partial\varphi\cap\IDD_n
\end{equation}
is well-defined and single-valued. 
Note that $x_0\in\IDD$. 
Suppose that $x_n\in\IDD$ for some $n\in\NN$. 
Then $x_n\in\IDD_n$, and hence, we deduce from 
\eqref{IIe:36b} that 
$x_{n+1}\in\dom\partial\varphi\cap\IDD_n\subset\IDD$. 
By reasoning by induction, we conclude that
\begin{equation}
\label{IIe:def}
(x_n)_{n\in\NN}\in\big(\IDD\big)^{\NN}
\quad\text{is well-defined}.
\end{equation}

\ref{IIt:1i}: 
We first derive from \eqref{IIal:1} and 
Lemma~\ref{IIl:prox}\ref{IIl:proxiii} that
\begin{equation}
\label{IIe:41}
(\forall n\in\NN)\quad\nabla f_n(x_n)-\nabla f_n(x_{n+1})
\in\gamma_n\partial\varphi(x_{n+1}).
\end{equation}
Next, by invoking \eqref{IIe:subdiff} and \eqref{IIe:41}, we get
\begin{equation}
(\forall x\in\dom\varphi\cap\dom f)(\forall n\in\NN)\quad
\gamma_n^{-1}\Pair{x-x_{n+1}}{\nabla f_n(x_n)-\nabla f_n(x_{n+1})}
+\varphi(x_{n+1})\leq\varphi(x).
\end{equation}
It therefore follows from \cite[Proposition~2.3(ii)]{IIBBC01} that
\begin{multline}
\label{IIe:50}
(\forall x\in\dom\varphi\cap\dom f)(\forall n\in\NN)\quad
\gamma_n^{-1}\big(D^{f_n}(x,x_{n+1})+D^{f_n}(x_{n+1},x_n)
-D^{f_n}(x,x_n)\big)\\
+\varphi(x_{n+1})\leq\varphi(x),
\end{multline}
and, in particular,
\begin{equation}
\label{IIe:42}
(\forall x\in\Argmin\varphi\cap\dom f)(\forall n\in\NN)\quad 
D^{f_n}(x,x_{n+1})\leq D^{f_n}(x,x_n)-D^{f_n}(x_{n+1},x_n).
\end{equation}
Since \eqref{IIe:36a} yields
\begin{equation}
\label{IIe:43}
(\forall x\in\Argmin\varphi\cap\dom f)(\forall n\in\NN)\quad 
D^{f_{n+1}}(x,x_{n+1})\leq (1+\eta_n)D^{f_n}(x,x_{n+1}),
\end{equation}
it follows from \eqref{IIe:42} that
\begin{multline}
\label{IIe:44}
(\forall x\in\Argmin\varphi\cap\dom f)(\forall n\in\NN)\quad 
D^{f_{n+1}}(x,x_{n+1})\leq(1+\eta_n)D^{f_n}(x,x_n)\\
-(1+\eta_n)D^{f_n}(x_{n+1},x_n).
\end{multline}
In particular,
\begin{equation}
\label{IIe:45}
(\forall x\in\Argmin\varphi\cap\dom f)(\forall n\in\NN)\quad 
D^{f_{n+1}}(x,x_{n+1})\leq (1+\eta_n)D^{f_n}(x,x_n).
\end{equation}
This shows that $(x_n)_{n\in\NN}$ is stationarily 
Bregman monotone with respect to $\Argmin\varphi$ 
relative to $(f_n)_{n\in\NN}$.

\ref{IIt:1ii}: 
Let $x\in\Argmin\varphi\cap\IDD$. It follows from \ref{IIt:1i} 
and Proposition~\ref{IIpp:qb1}\ref{IIpp:qb1i} that
\begin{equation}
\label{IIe:46}
\big(D^{f_n}(x,x_n)\big)_{n\in\NN}\quad\text{converges}
\end{equation}
and, since \eqref{IIe:44} yields
\begin{align}
(\forall n\in\NN)\quad 
D^{f_n}(x_{n+1},x_n)&\leq(1+\eta_n)D^{f_n}(x_{n+1},x_n)\nonumber\\
&\leq(1+\eta_n)D^{f_n}(x,x_n)-D^{f_{n+1}}(x,x_{n+1}),
\end{align}
we deduce that
\begin{equation}
\label{IIe:47}
D^{f_n}(x_{n+1},x_n)\to 0.
\end{equation}
On the other hand, since $(f_n)_{n\in\NN}$ 
is in $\BP_{\alpha}(f)$, we obtain
\begin{equation}
\label{IIe:48}
(\forall n\in\NN)\quad
\alpha D^f(x_{n+1},x_n)\leq D^{f_n}(x_{n+1},x_n).
\end{equation}
Altogether, \eqref{IIe:47} and \eqref{IIe:48} yield 
\begin{equation}
\label{IIe:49}
D^f(x_{n+1},x_n)\to 0.
\end{equation}
We also deduce from \eqref{IIe:50} that
\begin{equation}
\label{IIe:51}
(\forall n\in\NN)\quad
\varphi(x_{n+1})\leq\gamma_n^{-1}
\big(D^{f_n}(x_n,x_{n+1})+D^{f_n}(x_{n+1},x_n)\big)
+\varphi(x_{n+1})\leq\varphi(x_n).
\end{equation}
This shows that $(\varphi(x_n))_{n\in\NN}$ is decreasing, 
and hence, since it is bounded from below by $\inf\varphi(\XX)$, 
it converges. We now derive from \eqref{IIe:50} 
and \eqref{IIe:43} that
\begin{align}
\label{IIe:52}
(\forall n\in\NN)\quad &\dfrac{1}{\gamma}
\Bigg(\dfrac{1}{1+\eta_n}D^{f_{n+1}}(x,x_{n+1})
+D^{f_n}(x_{n+1},x_n)-D^{f_n}(x,x_n)\Bigg)
+\varphi(x_{n+1})\nonumber\\
&\leq\dfrac{1}{\gamma_n}
\Bigg(\dfrac{1}{1+\eta_n}D^{f_{n+1}}(x,x_{n+1})
+D^{f_n}(x_{n+1},x_n)-D^{f_n}(x,x_n)\Bigg)
+\varphi(x_{n+1})\nonumber\\
&\leq\varphi(x).
\end{align}
Hence, by using \eqref{IIe:46} 
and \eqref{IIe:47} after letting $n\to+\infty$ 
in \eqref{IIe:52}, we get
\begin{equation}
\inf\varphi(\XX)\leq\lim\varphi(x_n)
\leq\varphi(x)=\inf\varphi(\XX).
\end{equation}
In turn, $\varphi(x_n)\to\inf\varphi(\XX)$, i.e., 
$(x_n)_{n\in\NN}$ is therefore a minimizing sequence of $\varphi$. 

\ref{IIt:1iii}: 
We show actually that 
$\mathfrak{W}(x_n)_{n\in\NN}\subset\Argmin\varphi$. 
To this end, suppose that $x\in\mathfrak{W}(x_n)_{n\in\NN}$, 
i.e., $x_{k_n}\weakly x$. Since 
$\varphi$ is lower semicontinuous and convex, it is weakly 
lower semicontinuous \cite[Theorem~2.2.1]{IIZa02_B}, and hence,
\begin{equation}
\inf\varphi(\XX)\leq\varphi(x)\leq\varliminf
\varphi(x_{k_n})=\inf\varphi(\XX).
\end{equation}
In turn, $\varphi(x)=\inf\varphi(\XX)$, 
i.e., $x\in\Argmin\varphi$.

\ref{IIt:1iiia}:
Since $\XX$ is reflexive, we derive from \ref{IIt:1i} 
and Proposition~\ref{IIpp:qb1}\ref{IIpp:qb1ii} 
that $\mathfrak{W}(x_n)_{n\in\NN}\neq\emp$. Let us fix 
$\overline{x}\in\mathfrak{W}(x_n)_{n\in\NN}$. 
Since \eqref{IIe:def} yields 
$\mathfrak{W}(x_n)_{n\in\NN}\subset\Argmin\varphi
\cap\overline{\dom}f$, we get 
$\mathfrak{W}(x_n)_{n\in\NN}=\{\overline{x}\}$. 
In turn, $x_n\weakly\overline{x}$.

\ref{IIt:1iiib}:
We shall show that 
$\mathfrak{W}(x_n)_{n\in\NN}\subset\IDD$. 
To this end, let $\overline{x}\in\mathfrak{W}(x_n)_{n\in\NN}$, 
i.e., $x_{k_n}\weakly\overline{x}$. 
If $\Argmin\varphi\subset\IDD$ then 
$\overline{x}\in\Argmin\varphi\subset\IDD$. 
Now suppose that $\dom f^*$ is open and $\nabla f^*$ 
is weakly sequentially continuous. 
Let $x\in\Argmin\varphi\cap\IDD$. 
Then $\nabla f(x)\in\IDD^*$ \cite[Theorem~5.9]{IIBBC01} 
and it follows from \cite[Lemma~7.3(v)]{IIBBC01} that 
$D^{f^*}(\cdot,\nabla f(x))$ is coercive. 
Since $(D^f(x,x_{k_n}))_{n\in\NN}$ is bounded 
and since \cite[Lemma~7.3(vii)]{IIBBC01} asserts that
\begin{equation}
\label{IIe:27t}
(\forall n\in\NN)\quad D^{f^*}(\nabla f(x_{k_n}),\nabla f(x))
=D^f(x,x_{k_n}),
\end{equation}
we deduce that $(\nabla f(x_{k_n}))_{n\in\NN}$ is bounded. 
Take $\overline{x}^*\in\XX^*$ and a strictly increasing 
sequence $(p_{k_n})_{n\in\NN}$ in $\NN$ such that 
$\nabla f(x_{p_{k_n}})\weakly\overline{x}^*$. 
Since \cite[Lemma~7.3(ii)]{IIBBC01} states that 
$D^{f^*}(\cdot,\nabla f(x))$ is a proper lower semicontinuous 
convex function, we derive from \eqref{IIe:27t} that
\begin{equation}
D^{f^*}(\overline{x}^*,\nabla f(x))\leq\varliminf 
D^{f^*}\big(\nabla f(x_{p_{k_n}}),\nabla f(x)\big)
\leq\varliminf D^f(x,x_{p_{k_n}})<+\infty,
\end{equation}
which shows that $\overline{x}^*\in\dom f^*=\IDD^*$ 
and thus, by \cite[Theorem~5.10]{IIBBC01}, 
there exists $\overline{x}_1\in\IDD$ 
such that $\overline{x}^*=\nabla f(\overline{x}_1)$. 
Since $\nabla f^*$ is weakly sequentially continuous, 
we get
\begin{equation}
\overline{x}\leftharpoonup x_{p_{k_n}}
=\nabla f^*\big(\nabla f(x_{p_{k_n}})\big)
\weakly\nabla f^*(\overline{x}^*)=\overline{x}_1.
\end{equation} 
In turn, $\overline{x}=\overline{x}_1\in\IDD$. 
Finally, the claim follows from Proposition~\ref{IIpp:qb2}. 

\ref{IIt:1iv}: 
Since $\varphi\in\Gamma_0(\XX)$, $\Argmin\varphi$ is 
convex and closed, and the assertion therefore 
follows from Proposition~\ref{IIpp:qb3}.
\end{proof}

\begin{remark}
In Theorem~\ref{IIt:1}, suppose that $(\forall n\in\NN)$ 
$f_n=\hat{f}$, $\gamma_n=\gamma$, and $\eta_n=0$. 
Then \eqref{IIal:1} reduces to the Bregman proximal iterations 
\cite{IIBBC03}
\begin{equation}
(\forall n\in\NN)\quad x_{n+1}=\prox_{\gamma\varphi}^fx_n.
\end{equation}
\end{remark}

\subsection{An application to the convex feasibility problem}
\label{IIsec:conv}
\noindent
In this section, we apply the asymptotic analysis of 
variable Bregman monotone sequences to study 
the convex feasibility problem, i.e., the generic problem of 
finding a point in the intersection of a family of closed 
convex sets. We first recall the following results.

\begin{lemma}{\rm\cite[Definition~3.1 and Proposition~3.3]{IIBBC03}}
\label{IIle:Bclass}
Let $\XX$ be a reflexive real Banach space, let 
$f\in\Gamma_0(\XX)$ be G\^ateaux differentiable on 
$\IDD\neq\emp$, set
\begin{align}
\label{IIe:54}
(\forall(x,y)\in(\IDD)^2)\quad 
H^f(x,y)&=\menge{z\in\XX}{\Pair{z-y}{\nabla f(x)-\nabla f(y)}\leq 0}
\nonumber\\
&=\menge{z\in\XX}{D^f(z,y)+D^f(y,x)\leq D^f(z,x)}
\end{align}
and
\begin{multline}
\mathfrak{B}(f)=\Big\{T\colon\XX\to 2^{\XX}\;\Big|\;
\ran T\subset\dom T=\IDD\\
\text{and}\;
(\forall (x,y)\in\gra T)\;\Fix T\subset H^f(x,y)\Big\}.
\end{multline}
Let $T\in\mathfrak{B}(f)$ be such that $\Fix T\neq\emp$. 
Suppose that $f|_{\IDD}$ is strictly convex. 
Then the following hold:
\begin{enumerate}
\item
$\Fix T$ is convex.
\item
\label{IIle:Bclassii}
$(\forall x\in\overline{\Fix}T)(\forall (y,v)\in\gra T)$ 
$D^f(x,v)+D^f(v,y)\leq D^f(x,y)$.
\end{enumerate}
\end{lemma}

The class of operators $\mathfrak{B}$ includes types 
of fundamental operators  in Bregman optimization 
(see \cite{IIBBC03} for more discussions). 
We illustrate our result in Section~\ref{IIsec:VBMS} 
through an application to the problem 
of finding a common point of a family of closed convex 
subsets with nonempty intersection.

\begin{theorem}
\label{IIt:2}
Let $\XX$ be a reflexive real Banach space, 
let $I$ be a totally ordered at most countable index set, 
let $(C_i)_{i\in I}$ be a family of closed convex subsets 
of $\XX$ such that $C=\bigcap_{i\in I}C_i\neq\emp$, 
let $f\in\Gamma_0(\XX)$ be G\^ateaux differentiable on 
$\IDD\neq\emp$, let $(\eta_n)_{n\in\NN}\in\ell_+^1(\NN)$, 
let $\alpha\in\RPP$, and let $(f_n)_{n\in\NN}$ 
be Legendre functions in $\BP_{\alpha}(f)$ such that
\begin{equation}
\label{IIe:55}
(\forall n\in\NN)\quad 
(1+\eta_n)f_n\succcurlyeq f_{n+1}.
\end{equation} 
Let $\operatorname{i}\colon\NN\to I$ be such that
\begin{equation}
\label{IIe:56}
(\forall j\in I)(\exists M_j\in\NN\backslash\{0\})
(\forall n\in\NN)\quad j\in\{\operatorname{i}(n),\ldots,
\operatorname{i}(n+M_j-1)\}.
\end{equation}
For every $i\in I$, let $(T_{i,n})_{n\in\NN}$ be a sequence 
of operators such that
\begin{equation}
\label{IIe:57}
(\forall n\in\NN)\quad T_{i,n}\in\mathfrak{B}(f_n),
\quad C_i\cap\Fix T_{i,n}\neq\emp,\quad\text{and}\quad 
C_i\subset\overline{\Fix}T_{i,n}.
\end{equation}
Let $x_0\in\IDD$ and iterate
\begin{equation}
\label{IIal:2}
(\forall n\in\NN)\quad x_{n+1}\in T_{\operatorname{i}(n),n}x_n.
\end{equation}
Suppose that $f$ satisfies Condition~\ref{IIcd:2} 
and that $(\forall x\in\IDD)$ $D^f(x,\cdot)$ is coercive. 
Then there exists $\overline{x}\in C$ such that 
the following hold:
\begin{enumerate}
\item 
\label{IIt:2i}
Suppose that there exists $g\in\BF(f)$ 
that, for every $n\in\NN$, $g\succcurlyeq f_n$, and,
for every $x_1\in\XX$ and every $x_2\in\XX$,
\begin{equation}
\label{IIe:4q}
\begin{cases}
x_1\in\mathfrak{W}(x_n)_{n\in\NN}\cap C\\
x_2\in\mathfrak{W}(x_n)_{n\in\NN}\cap C\\
\big(\Pair{x_1-x_2}{\nabla f_n(x_n)}\big)_{n\in\NN}
\quad\text{converges}
\end{cases}
\Rightarrow\quad x_1=x_2,
\end{equation}
and that, for every strictly 
increasing sequence $(l_n)_{n\in\NN}$ in $\NN$, 
every $x\in\XX$, and every $j\in I$,
\begin{equation}
\label{IIe:59}
\begin{cases}
x_{l_n}\weakly x\\
y_{l_n}\in T_{j,l_n}x_{l_n}\\
y_{l_n}-x_{l_n}\to 0\\
(\forall n\in\NN)\; j=\operatorname{i}(l_n)
\end{cases}
\Rightarrow\quad x\in C_j.
\end{equation}
In addition, assume that $\mathfrak{W}(x_n)_{n\in\NN}\subset\IDD$. 
Then $x_n\weakly\overline{x}$.
\item 
\label{IIt:2ii}
Suppose that $f$ is Legendre, that $\varliminf D_C^f(x_n)=0$, 
and that there exists $\beta\in\RPP$ such that 
$(\forall n\in\NN)$ $\beta\hat{f}\succcurlyeq f_n$. 
Then $x_n\to\overline{x}$.
\end{enumerate}
\end{theorem}
\begin{proof}
For every $n\in\NN$ and every $i\in I$, 
we observe that $\ran T_{i,n}\subset\dom T_{i,n}=\IDD_n=\IDD$. 
Hence, it follows from \eqref{IIe:57} and \eqref{IIal:2} 
that $(x_n)_{n\in\NN}$ is a well-define sequence in 
$\IDD$. We now derive from \eqref{IIe:54}, \eqref{IIe:57}, 
and \eqref{IIal:2} that
\begin{equation}
\label{IIe:60}
(\forall x\in C\cap\dom f)(\forall n\in\NN)\quad 
D^{f_n}(x,x_{n+1})+D^{f_n}(x_{n+1},x_n)\leq D^{f_n}(x,x_n).
\end{equation}
Since \eqref{IIe:55} yields
\begin{equation}
\label{IIe:61}
(\forall x\in C\cap\dom f)(\forall n\in\NN)\quad 
D^{f_{n+1}}(x,x_{n+1})\leq (1+\eta_n)D^{f_n}(x,x_{n+1}),
\end{equation}
we deduce that
\begin{equation}
\label{IIe:62}
(\forall x\in C\cap\dom f)(\forall n\in\NN)\quad 
D^{f_{n+1}}(x,x_{n+1})\leq
(1+\eta_n)D^{f_n}(x,x_n)-(1+\eta_n)D^{f_n}(x_{n+1},x_n).
\end{equation}
In particular,
\begin{equation}
(\forall x\in C\cap\dom f)(\forall n\in\NN)\quad 
D^{f_{n+1}}(x,x_{n+1})\leq(1+\eta_n)D^{f_n}(x,x_n),
\end{equation}
which shows that $(x_n)_{n\in\NN}$ is stationarily Bregman 
monotone with respect to $C$ relative to $(f_n)_{n\in\NN}$. 
In addition, we derive from \eqref{IIe:57} that 
$(\forall i\in\{1,\ldots,m\})$ $C_i\cap\IDD\neq\emp$. 
Hence, $C\cap\IDD\neq\emp$.

\ref{IIt:2i}: 
In view of Proposition~\ref{IIpp:qb2}, it suffices to 
show that $\mathfrak{W}(x_n)_{n\in\NN}\subset C\cap\IDD$. 
To this end, let $\overline{x}\in\mathfrak{W}(x_n)_{n\in\NN}$, 
let $(k_n)_{n\in\NN}$ be a strictly increasing sequence in $\NN$ 
such that $x_{k_n}\weakly\overline{x}$, let $j\in I$, and 
let $x\in C\cap\IDD$. 
By \eqref{IIe:56}, there exists a strictly increasing sequence 
$(l_n)_{n\in\NN}$ in $\NN$ such that 
\begin{equation}
\label{IIe:63}
(\forall n\in\NN)\quad 
\begin{cases}
k_n\leq l_n\leq k_n+M_j-1<k_{n+1}\leq l_{n+1},\\
j=\operatorname{i}(l_n).
\end{cases}
\end{equation}
Since $D^f(x,\cdot)$ is coercive, it follows from 
Proposition~\ref{IIpp:qb1} that $(x_n)_{\in\NN}$ is bounded 
and $(D^{f_n}(x_{n+1},x_n))_{n\in\NN}$ converges. 
In turn, since \eqref{IIe:62} yields
\begin{align}
(\forall n\in\NN)\quad 
D^{f_n}(x_{n+1},x_n)
&\leq(1+\eta_n)D^{f_n}(x_{n+1},x_n)\nonumber\\
&\leq(1+\eta_n)D^{f_n}(x,x_n)-D^{f_{n+1}}(x,x_{n+1}),
\end{align}
we deduce that
\begin{equation}
\label{IIe:63a}
D^{f_n}(x_{n+1},x_n)\to 0.
\end{equation}
However, since
\begin{equation}
(\forall n\in\NN)\quad\alpha D^f(x_{n+1},x_n)\leq D^{f_n}(x_{n+1},x_n),
\end{equation}
it follows from \eqref{IIe:63a} that
\begin{equation}
D^f(x_{n+1},x_n)\to 0
\end{equation}
and hence, since $f$ satisfies Condition~\ref{IIcd:2},
\begin{equation}
\label{IIe:64}
x_{n+1}-x_n\to 0.
\end{equation}
Altogether, \eqref{IIe:63} and \eqref{IIe:64} imply that
\begin{equation}
\label{IIe:65}
\|x_{l_n}-x_{k_n}\|\leq\sum_{m=k_n}^{k_n+M_j-2}\|x_{m+1}-x_n\|
\leq (M_j-1)\max\limits_{k_n\leq m\leq k_n+M_j-2}\|x_{m+1}-x_m\|
\to 0,
\end{equation}
and therefore
\begin{equation}
\label{IIe:66}
x_{l_n}\weakly\overline{x}.
\end{equation}
Now let $(\forall n\in\NN)$ $y_{l_n}\in T_{j,l_n}x_{l_n}$. 
We deduce from \eqref{IIe:63} and \eqref{IIe:64} that 
\begin{equation}
\label{IIe:67}
y_{l_n}-x_{l_n}\to 0.
\end{equation}
By invoking successively \eqref{IIe:59}, \eqref{IIe:66}, 
and \eqref{IIe:67}, we get $\overline{x}\in C_j$, 
and hence, $\overline{x}\in C$. Consequently, 
$\mathfrak{W}(x_n)_{n\in\NN}\subset C\cap\IDD$.

\ref{IIt:2ii}: 
Since $C$ is closed, the assertion follows from 
Proposition~\ref{IIpp:qb3}.
\end{proof}

\begin{remark}\
\begin{enumerate}
\item
In Theorem~\ref{IIt:2}, suppose that 
$(\forall n\in\NN)$ $f_n=\hat{f}$ and $\eta_n=0$. 
Then we recover the framework of \cite[Section~4.2]{IIBBC03}.
\item
In Theorem~\ref{IIt:2}, suppose that $\XX$ is a Hilbert 
space, that $f=\|\cdot\|^2/2$, and that $(\forall n\in\NN)$ 
$f_n\colon x\mapsto\pair{x}{U_nx}/2$, where $(U_n)_{n\in\NN}$ 
are operators in $\BP_{\alpha}(\XX)$ such that 
$\sup_{n\in\NN}\|U_n\|<\pinf$ and $(\forall n\in\NN)$ 
$(1+\eta_n)U_n\succcurlyeq U_{n+1}$. 
Then we recover the version of 
\cite[Theorem~5.1(i) and (iii)]{IICV13b} 
without errors and $(\forall n\in\NN)$ $\lambda_n=1$.
\end{enumerate}
\end{remark}

Our last result concerns a periodic projection method 
that uses different Bregman distances at each iteration.

\begin{corollary}
\label{IIc:1}
Let $\XX$ be a reflexive real Banach space, 
let $m$ be a strictly positive integer, 
let $(C_i)_{1\leq i\leq m}$ be a family of closed convex 
subsets of $\XX$ such that $C=\bigcap_{i=1}^mC_i\neq\emp$, 
let $f\in\Gamma_0(\XX)$ be G\^ateaux differentiable on 
$\IDD$ such that $C\cap\IDD\neq\emp$, 
let $(\eta_n)_{n\in\NN}\in\ell_+^1(\NN)$, 
let $\alpha\in\RPP$, and let $(f_n)_{n\in\NN}$ 
be Legendre functions in $\BP_{\alpha}(f)$ such that
\begin{equation}
\label{IIe:68}
(\forall n\in\NN)\quad(1+\eta_n)f_n\succcurlyeq f_{n+1}.
\end{equation}
Let $x_0\in\IDD$ and iterate
\begin{equation}
\label{IIal:3}
(\forall n\in\NN)\quad 
x_{n+1}=P^{f_n}_{C_{1+\operatorname{rem}(n,m)}}x_n,
\end{equation}
where $\operatorname{rem}(\cdot,m)$ is the remainder of 
the division by $m$. Suppose that $f$ satisfies 
Condition~\ref{IIcd:2} and that $(\forall x\in\IDD)$ 
$D^f(x,\cdot)$ is coercive. Then there exists 
$\overline{x}\in C$ such that the following hold:
\begin{enumerate}
\item
\label{IIc:1i}
Suppose that there exists $g\in\BF(f)$ 
such that, for every $n\in\NN$, $g\succcurlyeq f_n$, 
and, for every $x_1\in\XX$ and every $x_2\in\XX$,
\begin{equation}
\label{IIe:4s}
\begin{cases}
x_1\in\mathfrak{W}(x_n)_{n\in\NN}\cap C\\
x_2\in\mathfrak{W}(x_n)_{n\in\NN}\cap C\\
\big(\Pair{x_1-x_2}{\nabla f_n(x_n)}\big)_{n\in\NN}
\quad\text{converges}
\end{cases}
\Rightarrow\quad x_1=x_2.
\end{equation}
In addition, suppose that $\mathfrak{W}(x_n)_{n\in\NN}\subset\IDD$. 
Then $x_n\weakly\overline{x}$.
\item
\label{IIc:1ii}
Suppose that $f$ is Legendre, that $\varliminf D_C^f(x_n)=0$, 
and that there exists $\beta\in\RPP$ such that 
$(\forall n\in\NN)$ $\beta\hat{f}\succcurlyeq f_n$. 
Then $x_n\to\overline{x}$.
\end{enumerate}
\end{corollary}
\begin{proof}
First, we see that the function 
$\operatorname{i}\colon\NN\to\{1,\ldots,m\}\colon 
n\mapsto 1+\operatorname{rem}(n,m)$ satisfies \eqref{IIe:56}, 
where $(\forall j\in\{1,\ldots,m\})$ $M_j=m$. Now set 
\begin{equation}
(\forall i\in\{1,\ldots,m\})(\forall n\in\NN)\quad 
T_{i,n}=P_{C_i}^{f_n}.
\end{equation}
Then, by \cite[Theorem~3.34]{IIBBC03}, for every $n\in\NN$ 
and every $i\in\{1,\ldots,m\}$, we have
\begin{equation}
T_{i,n}\in\mathfrak{B}(f_n)\quad\text{and}\quad 
C_i\cap\overline{\dom}f\cap\Fix T_{i,n}=C_i\cap\IDD
\supset C\cap\IDD\neq\emp.
\end{equation}
In addition, it follows from \cite[Lemma~3.2]{IIBBC03} that
\begin{equation}
(\forall n\in\NN)(\forall i\in\{1,\ldots,m\})\quad 
C_i\cap\overline{\dom f}=\overline{C_i\cap\IDD}
=\overline{C_i\cap\IDD_n}=\overline{\Fix}T_{i,n}.
\end{equation}
Therefore, \eqref{IIal:3} is a particular case of 
\eqref{IIal:2}. We shall actually apply Proposition~\ref{IIt:2} 
with the family $(C_i\cap\overline{\dom}f)_{1\leq i\leq m}$.

\ref{IIc:1i}: 
Let us fix $j\in\{1,\ldots,m\}$ and suppose that 
\begin{equation}
x_{l_n}\weakly x,\quad T_{j,l_n}x_{l_n}-x_{l_n}\to 0,
\quad\text{and}\quad (\forall n\in\NN)\quad 
j=\operatorname{i}(l_n).
\end{equation} 
Then $C_j\ni P_{C_j}^{f_{l_n}}x_{l_n}=T_{j,l_n}x_{l_n}
\weakly x$, and hence, $x\in C_j$ since $C_j$ is weakly 
closed \cite[Corollary~4.5]{IISimo08}. Moreover, 
since $(x_n)_{n\in\NN}$ is in $\IDD$, $x\in\overline{\dom}f$ 
and hence $x\in C_j\cap\overline{\dom}f$. 
This shows that \eqref{IIe:59} is satisfied. Consequently, 
the assertion follows from Proposition~\ref{IIt:2}\ref{IIt:2i}.

\ref{IIc:1ii}: 
We have
\begin{equation}
(\forall n\in\NN)\quad\inf_{x\in C\cap\overline{\dom}f}D^f(x,x_n)
\leq\inf_{x\in C\cap\dom f}D^f(x,x_n)=D_C^f(x_n),
\end{equation}
and hence, $\varliminf D_{C\cap\overline{\dom}f}(x_n)=0$. 
The claim therefore follows from 
Proposition~\ref{IIt:2}\ref{IIt:2ii}.
\end{proof}
\vskip 0.5cm
\noindent{\bf Acknowledgment.} 
I would like to thank my doctoral advisor Professor Patrick L. 
Combettes for bringing this problem to my attention and for 
helpful discussions.

\end{document}